\documentclass[letterpaper, 10 pt, conference]{ieeeconf}  

\IEEEoverridecommandlockouts                              

\usepackage{amsmath}
\usepackage{amsmath,amssymb,amsfonts} 
\usepackage{graphicx}
\usepackage{algorithm}
\usepackage{algpseudocode}
\usepackage{hyperref}
\hypersetup{hidelinks=true}
\usepackage{textcomp}
\usepackage{epstopdf}
\usepackage{booktabs} 
\usepackage{hyperref}
\usepackage{amsfonts}       
\usepackage{multicol}
\usepackage{multirow} 
\usepackage{subcaption}
\usepackage{xcolor}

\usepackage{amsthm}

\newcommand{\R}{\mathbb{R}}

\newcommand{\bmat}[1]{\begin{bmatrix}#1\end{bmatrix}}

\newcommand{\ip}[2]{\langle #1, #2 \rangle}
\newcommand{\mcl}[1]{\mathcal{#1}}
\newcommand{\mbf}[1]{\mathbf{#1}}

\newcommand{\C}{\mathbb{C}}
\newcommand{\N}{\mathbb{N}}

\theoremstyle{plain}

\newtheorem{thm}{Theorem}

\newtheorem{lem}[thm]{Lemma}

\newtheorem{cor}[thm]{Corollary}

\title{\LARGE \bf
Stable Rational Approximation of PDE Transfer Functions with $H_\infty$ Error Bounds
}

\author{Aleksandr Talitckii and Matthew Peet$^1$
\thanks{This work was supported by the National Science Foundation under
grants NSF CCF-2323532}
\thanks{$^{1}$Aleksandr Talitckii and Matthew M. Peet are with the School for the Engineering of Matter, Transport and Energy, Arizona State University, Tempe, AZ, 85298, USA
        {\tt\small atalitck@asu.edu} and {\tt\small mpeet@asu.edu}} 
}

\begin{document}

\maketitle
\thispagestyle{empty}
\pagestyle{empty}

\begin{abstract}
 Transfer functions of Partial Differential Equations (PDEs) are irrational and difficult to obtain. Unlike for rational transfer functions, only a few methods are available for the control and analysis of irrational transfer functions. Thus, for robust analysis and control of PDEs, we need to use a rational approximation of the transfer function, preferably with provable $H_\infty$ error bounds. In this paper, we extend Krylov subspace model reduction to infinite-dimensional systems modeled by Partial Integral Equations (PIEs) -- an equivalent representation of PDEs. First, we propose a state transformation of the PIE system that allows us to compute a controllability matrix. Then, to obtain a rational approximation, we project the PIE states onto the vector space spanned by the truncated controllability matrix. Next, we show that sufficient conditions for exponential stability and well-posedness of the PIE system guarantee the exponential stability of the reduced-order system and provide constructive finite-frequency $H_\infty$ error bounds. Finally, we provide four illustrative numerical examples to demonstrate the performance of the proposed method.

\end{abstract}

\section{Introduction}\label{sec:introduction}


The transfer function provides a fundamental link between the state-space representation of a linear system and the system input-output behavior. For Ordinary Differential Equations (ODEs), transfer functions are rational, and there exist extensive toolsets for their analysis and control~\cite{desoer2009feedback}. Unfortunately, transfer functions for Partial Differential Equations (PDEs) are irrational, and methods for rational transfer functions do not extend readily to PDE transfer functions.  This motivates the development of methods for rational approximation of transfer functions for PDEs.

Although numerous methods exist for obtaining rational transfer function approximations for PDEs, these methods do not typically admit constructive a priori error bounds (i.e., in the $H_\infty$ sense)~\cite{curtain2009transfer}. Such error bounds are necessary when robust analysis and control methods are used to certify stability and design controllers with provable performance based on the rational approximation. That is, the unmodeled dynamics become a norm-bounded uncertainty with norm bounded by the truncation error in the rational approximation. Thus, a rational approximation of the PDE transfer function with constructive a priori error bounds is required for robust analysis and control methods.

Before turning to the literature on rational approximation, let us first examine cases when analytic representations of irrational transfer functions exist and are possible to obtain. The Laplace transform can be used to show the existence of the transfer functions for LTI systems~\cite{curtain2020introduction}. Then, to obtain such transfer functions of PDEs, we need to find an inverse of a differential operator by solving a two-point boundary value problem~\cite{ jovanovic2003exact}. However, the derivation often becomes intractable for PDEs because of spatially varying coefficients and complex boundary conditions. Therefore, analytic transfer functions are available only for a few special cases~\cite{zwart2004transfer}.
 
Even when analytic representations of an irrational transfer function are available, such expressions are not compatible with the well-developed toolset for analysis and control of rational and state-space systems. The most common alternative, then, is to discretize the PDE by projecting solutions onto a finite-dimensional basis such as the Chebyshev polynomials. Although the projection allows us to construct rational approximations,  discretization approaches may lose dissipativity properties~\cite{harkort2012stability} or  introduce spurious modes through the numerical treatment of boundary conditions~\cite{karam2024implicit}. Hence, the discretized system may not be stable even if the original PDE is stable. Furthermore, without additional assumptions, $H_\infty$ error bounds on the unmodeled dynamics are generally unavailable~\cite{selivanov2025improved}, limiting the applicability of the result to robust control and stability analysis.

A possible way to construct a rational approximation with $H_\infty$ error bounds is to generalize finite-dimensional model reduction methods to infinite-dimensional systems. A generalization of balanced model truncation approaches may provide theoretical $H_\infty$ error bounds under suitable assumptions~\cite{guiver2014model}. However, computing the balanced realization is difficult for infinite-dimensional systems. Thus, practical implementations require Galerkin or finite-element discretization methods~\cite{Singler2011BalancedPF}, so the convergence of such discretizations requires separate analysis. To avoid computation of the balanced realization, we may consider Krylov subspace model reduction~\cite{freund2003model}. For ODEs, the Krylov subspace model reduction uses a truncated controllability or observability matrix to project the states onto the low-dimensional subspace. However, infinite-dimensional Krylov subspace model reduction methods~\cite{harkort2010krylov} do not generally guarantee stability of the reduced-order model or provide $H_\infty$ error bounds. 

To construct a stable rational approximation of PDE transfer functions with provable finite-frequency $H_\infty$ bounds, we extend Krylov subspace model reduction to systems modeled by Partial Integral Equations (PIEs). Under suitable boundary conditions, the PIE formulation creates a bijection between the PDE states, $\mbf v(t, \theta)$, and the PIE states, $\mbf x(t,\theta):= \partial_{\theta}^n \mbf v(t, \theta)$, defined by a Partial Integral (PI) operator  $\mcl T$ as $\mbf v(t, \theta) = (\mcl T \mbf x)(t, \theta)$~\cite{peet2021partial}.

Then, the input-output PIE system has the form \vspace{-1.5mm} 
\begin{equation}\label{eqn:PIE_intro}
    \partial_t(\mcl T \mbf x)(t)\hspace{-0.7mm} =\hspace{-0.7mm}\mcl A\mbf x(t)\hspace{-0.7mm}+\hspace{-0.7mm}\mcl Bu(t), \quad  y(t) \hspace{-0.7mm}=\hspace{-0.7mm} \mcl C \mbf x(t) \hspace{-0.7mm}+\hspace{-0.7mm} \mcl D u(t),\vspace{-1.5mm}
\end{equation}
where $\mcl T, \mcl A, \mcl B,\mcl C,\mcl D$ are PI operators, $u(t)$ and $y(t)$ are the inputs and outputs.

The transfer function of the PIE system in~\eqref{eqn:PIE_intro} can be defined as $\hat G(s) =\mcl C(s\mcl T - \mcl A)^{-1}\mcl B + \mcl D$. 
However, applying  Krylov subspace model reduction directly to the PIE in~\eqref{eqn:PIE_intro} is still challenging, since it includes the inverse of $s \mcl T - \mcl A$ and the controllability matrix cannot be constructed without inverting $\mcl T$.

To address these challenges, we choose some $\sigma \in \R$ such that $\sigma \mcl T - \mcl A$ is invertible. Then, using the variable substitution $\mbf z(t):= (\mcl A \hspace{-0.2mm}-\hspace{-0.2mm} \sigma \mcl T) \mbf x(t)$, we transform the PIE system into\vspace{-1.5mm}
\begin{align*}
    \partial_t(\mcl R \mbf z)(t) \hspace{-0.9mm}=\hspace{-0.9mm}(I \hspace{-0.9mm}+\hspace{-0.9mm} \sigma\mcl R)\mbf z(t)\hspace{-0.9mm}+\hspace{-0.9mm}\mcl Bu(t),\quad  y(t)\hspace{-0.9mm}=\hspace{-0.9mm} \hat{\mcl C} \mbf z(t) \hspace{-0.9mm}+\hspace{-0.9mm} \mcl D u(t),\\[-7mm]
\end{align*}
where $\mcl R:= \mcl T(\mcl A - \sigma \mcl T)^{-1}$ and $\hat{\mcl C}:= \mcl C(\mcl A-\sigma \mcl T)^{-1}$.

The transformed system allows us to construct the truncated controllability matrix as $C_S = \bmat{\mcl B, \mcl R\mcl B, \dots, \mcl R^{r-1} \mcl B}$. Then, projecting the PIE states onto the vector space spanned by the columns of $C_S$ allows us to propose a rational approximation of the PIE transfer function.  Next, using sufficient conditions for well-posedness and  exponential stability of the PIE,  we derive finite-frequency $H_\infty$ error bounds and show that the reduced-order system is stable.

We organize the paper as follows. First, we formulate Partial Integral Equations (PIEs) and discuss the transfer functions of PIEs in Sec.~\ref{sec:PIE_problem_formulation}. Then, in Sec.~\ref{sec:transformed_PIE}, we present the transformed PIEs. Next, in Sec.~\ref{sec:Arnoldi_iterations}, we propose a rational approximation of the transfer function. In Sec.~\ref{sec:error_bounds}, we derive $H_\infty$  error bounds.  In Sec.~\ref{sec:PDE_transfer_function}, we show a relationship between PDE and PIE transfer functions. We conclude the paper by presenting numerical examples in Sec.~\ref{sec:numerical_examples}.

 \subsection*{Notation}
We use $\N,\R,\C$ to denote the natural numbers, the real numbers and the complex numbers. For $i,j \in \N$, we use $\delta_{ij}$ to denote the Kronecker delta function -- $\delta_{ij} = 1$ if $i = j$, and $0$ otherwise. For $x\in\C^n$, we use $x^\dagger$ to denote the conjugate transpose of $x$. For $\Omega \subseteq \R^n$, we denote by $L^n_2(\Omega, \C)$  the space of  Lebesgue square-integrable complex-valued functions $f:
\Omega\rightarrow \C^n$ with  inner product $\ip{f}{g} = \int_\Omega f^\dagger(s) g(s) ds$. For a given Hilbert space $H$, we denote by $I:H\rightarrow H$ and $0:H\rightarrow H$ the identity operator and the null operator. For given $a, b\in \R$ and $n_1,n_2 \in \N$, we denote by $\mbf Z^{n_1, n_2}$ the Hilbert space $\mbf Z^{n_1, n_2}:= \C^{n_1} \times L_2([a, b], \C^{n_2})$ with  inner product $\ip{\bmat{x_1 \\ \mbf x_2}}{ \bmat{y_1 \\ \mbf y_2}}_{\mbf Z} := x_1^\dagger y_1 + \ip{\mbf x_2}{\mbf y_2}_{L_2}$.

\section{Transfer Function of Partial Integral Equations}
\label{sec:PIE_problem_formulation}

In this paper, we consider an input-output map modeled by complex-valued Partial Differential Equations, which can be represented by Partial Integral Equations (PIEs). We initially consider the problem of approximating transfer functions of  PIE systems, which allows us to derive $H_\infty$ error bounds. Then, the application of the results to PDE transfer functions is discussed in Sec.~\ref{sec:PDE_transfer_function}. Specifically, consider a PIE system \vspace{-1.5mm}
\begin{align}\label{eqn:PIE}
    \partial_t(\mcl T \mbf x)(t) & = \mcl A \mbf x(t) + \mcl B u(t), \qquad \mbf x(0)= 0 \\
    y(t) & = \mcl C \mbf x(t) + \mcl Du(t), \notag \\[-7mm]\notag
\end{align}
where $\mbf x(t) \in \mbf Z:=  \C^n \times L_2([a, b], \C^m)$ is the state, $u(t) \in \C$, $y(t) \in \C$ are the inputs and outputs of the system, $\mcl T, \mcl A:\mbf Z\rightarrow\mbf Z$, $\mcl C:\mbf Z\rightarrow \C$, $\mcl B \in \mbf Z$ and $\mcl D \in \C$ are complex-valued Partial Integral (PI) operators.

Complex-valued PI operators are a slight modification of real-valued PI operators discussed in~\cite{peet2021partial}. Specifically, given  a matrix $P \in \C^{n\times n}$, bounded complex-valued functions  $Q_1(\theta), Q_2(\theta), R_0(\theta)$ and $R_1(\theta, \nu)$, we say $\mcl P$ is a Partial Integral (PI) operator on $\mbf Z$ if \vspace{-3.5mm}

{\small
\begin{equation}\label{eqn:PIoperator}
\bigg(\mcl P \hspace{-1mm}\bmat{P\hspace{-1mm} &\hspace{-1mm} Q_1 \\ Q_2 \hspace{-1mm}&\hspace{-1mm} \{R_i\}}\hspace{-1mm} \bmat{x \\ \Phi}\hspace{-1mm}\bigg)(\theta) = \bmat{P x + \int_a^b Q_1(\theta)\Phi(\theta)d\theta \\ 
Q_2(\theta)x + (\mcl P_{R_i} \Phi)(\theta) },\vspace{-1.5mm}
\end{equation}}
where $(\mcl P_{R_i} \Phi)(\theta) = R_0(\theta)\Phi(\theta) + \int_a^b R_1(\theta, \nu)\Phi(\nu) d\nu.$

We define the transfer function of the PIE system  as\vspace{-1.5mm}
\begin{equation}\label{eqn:transferfunction_PIE}
    \hat G(s) = \mcl C(s\mcl T - \mcl A)^{-1} \mcl B + \mcl D.\vspace{-1.5mm} 
\end{equation}
Our goal is to approximate the transfer function for $s \in \{s \in \C\; |\; Re(s) > -\omega\}$, where $\omega$ is an exponential decay rate of the $C_0$-semigroup generated by $A:=\mcl A \circ \mcl T^{-1}$. The particular choice of $\omega$ and a connection between the PIE and PDE transfer functions are discussed in Sec.~\ref{sec:PDE_transfer_function}.

\section{Transformed Partial Integral Equations} \label{sec:transformed_PIE}

In this section, we propose a transformation of the PIE system in~\eqref{eqn:PIE}, which allows us to construct a controllability matrix and to reformulate the inverse of $(s \mcl T - \mcl A)$ as a problem of finding the resolvent of a bounded linear operator $\mcl R$. For this task, we assume that there exists $\sigma \in \R$ such that $\mcl A - \sigma \mcl T $ is invertible -- a condition for well-posedness test of the PIE system in~\cite{jagt2026verifying}. 

Using the invertibility of $\mcl A - \sigma \mcl T $, we may define the variable substitution as $\mbf z(t) = (\mcl A - \sigma \mcl T) \mbf x(t)$. Then, the variable substitution transforms the PIE system in~\eqref{eqn:PIE} to \vspace{-1.5mm}
\begin{align}\label{eqn:PIE_transformed}
    \partial_t(\mcl R \mbf z)(t) & = (I + \sigma \mcl R)\mbf z(t) + \mcl B u(t) \\
    y(t) & = \hat{\mcl C} \mbf z(t) + \mcl Du(t), \notag\\[-7mm]\notag
\end{align}
where $\mcl R:= \mcl T(\mcl A - \sigma \mcl T)^{-1}$ and $\hat{\mcl C} := \mcl C(\mcl A- \sigma \mcl T)^{-1}$,
and where we use the fact that\vspace{-1.5mm}
\begin{align*}
\mcl A\mbf x(t) & =\mcl A(\mcl A -\sigma \mcl T)^{-1} \mbf z(t) = (I + \sigma \mcl T(\mcl A - \sigma\mcl T)^{-1})\mbf z(t)\\
\mcl T \mbf x(t) &= \mcl T(\mcl A - \sigma \mcl T)^{-1}\mbf z(t) = \mcl R \mbf z(t). \\[-7mm]
\end{align*}
The transfer function of the PIE in Eqn.~\eqref{eqn:PIE_transformed} becomes\vspace{-1.5mm}
\begin{equation}\label{eqn:transferfunction_transformed_PIE}
    \hat G(s) = \hat{\mcl C}((s - \sigma)\mcl R - I)^{-1} \mcl B + \mcl D\vspace{-1.5mm}
\end{equation}
The equivalence between the transfer functions in~\eqref{eqn:transferfunction_PIE} and in~\eqref{eqn:transferfunction_transformed_PIE} is established as follows. 
\begin{lem}\label{lem:transformed_PIE_equivalence}
    Let $\mbf X$ be a complex Hilbert space and $\mcl A:\mbf X\rightarrow \mbf X$, $\mcl B:\C\rightarrow \mbf X$, $\mcl C:\mbf X\rightarrow\C$, $\mcl D:\C\rightarrow\C$ and $\mcl T:\mbf X\rightarrow \mbf X$ be bounded linear operators. Suppose there exists $\sigma \in \R$ such that $\mcl A - \sigma \mcl T$ has a bounded inverse.    
    Then, for any $s \in \C$ such that $s \mcl T - \mcl A$ is invertible, we have \vspace{-1.5mm}
    \[
    \mcl C(s \mcl T - \mcl A)^{-1}\mcl B + \mcl D = \hat{\mcl C}((s - \sigma)\mcl R - I)^{-1} \mcl B + \mcl D, \vspace{-1.5mm}
    \]
where $\mcl R:= \mcl T(\mcl A - \sigma \mcl T)^{-1}$ and $\hat{\mcl C} := \mcl C(\mcl A- \sigma \mcl T)^{-1}$.
    
\end{lem}\vspace{-1.5mm}
\begin{proof}
    In this proof, we show that $C(s \mcl T - \mcl A)^{-1} =  \hat{\mcl C} ((s - \sigma)\mcl R - I) ^{-1}$.   
    
    For all $s \in \C$, we have that \vspace{-1.5mm}
    \begin{align*}
    s\mcl T \hspace{-0.35mm}-\hspace{-0.35mm} \mcl A& = (s\hspace{-0.35mm}-\hspace{-0.35mm}\sigma)\mcl T\hspace{-0.35mm} -\hspace{-0.35mm}(\mcl A \hspace{-0.35mm}-\hspace{-0.35mm} \sigma \mcl T) =((s\hspace{-0.35mm}-\hspace{-0.35mm}\sigma)\mcl R \hspace{-0.35mm}-\hspace{-0.35mm} I)(\mcl A \hspace{-0.35mm}-\hspace{-0.35mm} \sigma \mcl T). \\[-7mm]
    \end{align*}
    Next, if $s \mcl T - \mcl A$ is invertible, we have \vspace{-1.5mm}
    \[
    (s \mcl T - \mcl A)^{-1} = (\mcl A - \sigma \mcl T)^{-1} ((s - \sigma)\mcl R - I)^{-1}. \vspace{-1.5mm}
    \]
    Thus, since $\hat{\mcl C} := \mcl C(\mcl A- \sigma \mcl T)^{-1}$, we have \vspace{-1.5mm}
    \[
    \mcl C(s \mcl T - \mcl A)^{-1}\mcl B + \mcl D =   \hat{\mcl C}((s - \sigma)\mcl R - I)^{-1} \mcl B + \mcl D.\vspace{-1.5mm}
    \]
    for all $s \in \C$ such that $s \mcl T - \mcl A$ is invertible.
\end{proof}\vspace{-1.5mm}
The advantage of the transfer function in Eqn.~\eqref{eqn:transferfunction_transformed_PIE} is that the descriptor and state operators are described by $\mcl R$. However, even if $\mcl A, \mcl T$ are PI operators, the inverse of $\mcl A - \sigma \mcl T$ may not be a PI operator. In this case, we should approximate the inverse from a set of PI operators -- See Sec.~\ref{sec:PDE_transfer_function}. 
 
\section{A Rational Approximation of PIE Transfer functions }\label{sec:Arnoldi_iterations}

In this section, we propose a rational approximation of the transfer function in  Eq.~\eqref{eqn:transferfunction_PIE}. We accomplish the task in two steps. First, in Subsec.~\ref{subsec:resolvent_approximation}, we consider the problem of approximating the action of the shifted resolvent operator of a bounded linear operator $\mcl R: \mbf X\rightarrow \mbf X$ on an element $\mcl B \in \mbf X$. For this task, we use the subspace spanned by a controllability matrix of the PIE system in~\eqref{eqn:PIE_transformed}. Next, using the resolvent approximation, in Subsec.~\ref{subsec:rational_approximation}, we propose a rational approximation of the transfer function in~\eqref{eqn:transferfunction_transformed_PIE}.

 \subsection{ Controllability Subspace for a Resolvent Approximation}\label{subsec:resolvent_approximation}
Let us start by the resolvent approximation. For this task, we consider the controllability matrix for the  PIE system in~\eqref{eqn:PIE_transformed}. Recall that for the state operator $\mcl R: \mbf X\rightarrow \mbf X$ and the input operator $\mcl B \in \mbf X$, the controllability matrix is defined as \vspace{-1.5mm}
\begin{equation}\label{eqn:controllabilitymatrix}
C_S = \bmat{\mcl B & \mcl R \mcl B & \dots & \mcl R^n \mcl B & \dots}.\vspace{-1.5mm}
\end{equation}
Unlike for finite-dimensional systems, the controllability matrix in~\eqref{eqn:controllabilitymatrix} may have an infinite number of linearly independent columns. Due to computational limitations, we consider only a finite number of columns of the controllability matrix.

The first columns of the operator in~\eqref{eqn:controllabilitymatrix} define the Krylov subspace spanned by operators $\mcl R$ and $\mcl B$\vspace{-1.5mm}
\begin{equation}\label{eqn:Krylov_subspace}
    \mcl K_r(\mcl R, \mcl B) = \texttt{span}\{\mcl B, \mcl R\mcl B, \dots , \mcl R^{r-1}\mcl B\}.\vspace{-1.5mm}
\end{equation} 
Krylov subspace methods can construct an approximation of operator functions. Specifically, let $\mbf q_1, \dots, \mbf q_r$ be an orthonormal basis of $\mcl K_r(\mcl R, \mcl B)$. 
Then, the approximation of $(I - s\mcl R)^{-1} \mcl B$ is given as  \vspace{-1.5mm}
\begin{equation}\label{eqn:resolvent_approximation}
(I - s\mcl R)^{-1} \mcl B \approx \mbf V (I - sH)^{-1} \mbf V^*\mcl B,\vspace{-1.5mm}
\end{equation}
where $\mbf V :\C^r \rightarrow \mbf X$ and $H \in \C^{r \times r}$ are defined as \vspace{-1.5mm}
\begin{equation}\label{eqn:H_VRV}
H_{ij} = \ip{\mbf q_i}{\mcl R \mbf q_j}_{\mbf X} \qquad \mbf V \alpha := \sum\nolimits_{i = 1}^r \alpha_i \mbf q_i \vspace{-1.5mm} 
\end{equation} 

The error bounds for the approximation are established as follows. 
\begin{thm}[Thm 4.10 in~\cite{guttel2010rational}]\label{thm:function_approximation_KS} Let $\mbf X$ be a complex Hilbert space, $\mcl B \in \mbf X$ and $\mcl R:\mbf X \rightarrow \mbf X$ be a bounded linear operator. Define \vspace{-4mm}
\begin{equation}\label{eqn:Sigma_set}
    \Sigma = cl\big(\{\ip{\mbf x}{\mcl R \mbf x}_{\mbf X} \; | \; \mbf x \in \mbf X\; ,\|\mbf x\|_{\mbf X} = 1\}\big).\vspace{-1.5mm}
\end{equation}
Let $\{\mbf q_i\}_{i = 1}^r \subset \mbf X$ be an orthonormal basis of $\mcl K_r(\mcl R, \mcl B )$ and $H, \mbf V$ be defined as in~\eqref{eqn:H_VRV}.

Then, there exists $K \geq 0$ such that for all $s \in \C$ if $s^{-1} \hspace{-0.5mm} \notin \hspace{-0.5mm}  \Sigma$\vspace{-1.5mm}
\begin{align}
&\|(I - s\mcl R)^{-1} \mcl B - \mbf V (I - s H)^{-1} \mbf V^* \mcl B\|_{\mbf X} \\[-1mm]
&\qquad\qquad \leq K\|\mcl B\|_{\mbf X} \min_{p \in \C_r[z]} \max_{z \in \Sigma} \big| \frac{1}{1 - s z} - p (z)\big|,\notag\\[-7mm]\notag
\end{align}
where $\C_r[z] = \{p(z) = \sum_{i = 0}^{r- 1} \alpha_i z^{i}\; |\; \alpha_i \in \C\}$.
\end{thm}\vspace{-1.5mm} 
\begin{proof}
    Since a function $f(z) := (1 - sz)^{-1}$ is analytic on $\Sigma$ if $s^{-1} \notin \Sigma$, the proof of the theorem is an application of Thm. 4.10 in~\cite{guttel2010rational}.
\end{proof}\vspace{-1.5mm}

Thm.~\ref{thm:function_approximation_KS} shows that if $s^{-1} \notin \Sigma$, the Krylov subspace approximation in~\eqref{eqn:resolvent_approximation} finds a polynomial approximation of degree $r-1$ based on a basis of $\mcl K_r(\mcl R, \mcl B)$.  In the next subsection, we generalize Thm.~\ref{thm:function_approximation_KS} to approximate the transfer function in~\eqref{eqn:transferfunction_transformed_PIE}. 
 
\subsection{A Rational Approximation of Transfer Functions}\label{subsec:rational_approximation}

To approximate the transfer functions in~\eqref{eqn:transferfunction_PIE} and~\eqref{eqn:transferfunction_transformed_PIE}, we use the results of Thm.~\ref{thm:function_approximation_KS}. Specifically, by the resolvent approximation in Eq.~\eqref{eqn:resolvent_approximation}, we may define an approximation of the transfer function in~\eqref{eqn:transferfunction_transformed_PIE} as \vspace{-1.5mm}
\begin{equation}
   \hat G_r(s) = \hat{\mcl C} \mbf V((s - \sigma) H - I)^{-1} \mbf V^* \mcl B + \mcl D, \label{eqn:transfer_function_approximation_PIE_transformed}\vspace{-1.5mm}
\end{equation}
where recall $\hat{\mcl C} = \mcl C(\mcl A - \sigma \mcl T)^{-1}$, $\mcl R =\mcl T(\mcl A - \sigma \mcl T)^{-1}$ and $H, \mbf V$ are defined as in~\eqref{eqn:H_VRV}. 
 
\begin{thm}\label{thm:transfer_function_error_bounds_without_stability}
     Let $\mbf X$ be a complex Hilbert space, $\mcl D\in\C$, $\mcl B \in \mbf X$ and $\mcl T, \mcl A:\mbf X \rightarrow \mbf X$, $\mcl C:\mbf X \rightarrow \C$ be bounded linear operators. Suppose there exists $\sigma \in \C$ such that $\mcl A - \sigma \mcl T$ has a bounded inverse.    
     Let $\mcl R = \mcl T(\mcl A - \sigma \mcl T)^{-1}$ and $\{\mbf q_i\}_{i = 1}^r$ be an orthonormal basis of $\mcl K_r(\mcl R, \mcl B )$. Let $H, \mbf V$ be defined as in~\eqref{eqn:H_VRV}, $\hat G(s)$ as in~\eqref{eqn:transferfunction_PIE} and $\hat G_r(s)$ as in~\eqref{eqn:transfer_function_approximation_PIE_transformed}.
     
     Then, there exists $K \geq 0$ such that \vspace{-2.5mm}
     \begin{align*}
     &|\hat G_r(s) - \hat G(s)| \\[-3mm]
     &\leq K \|\mcl C(\mcl A - \sigma \mcl T)^{-1} \| \|\mcl B\|\hspace{-2mm}\min_{p \in \C_r[z]}  \max_{z \in \Sigma} \bigg| \frac{1}{1 - (s-\sigma) z} - p (z)\bigg|,\\[-7mm]
     \end{align*}
     for all $s \in \{s \in \C\;| \; (s-\sigma)^{-1} \notin \Sigma\}$.
\end{thm}\vspace{-1.5mm}
\begin{proof}
 The proof of this theorem is an application of Lem.~\ref{lem:transformed_PIE_equivalence} and Thm.~\ref{thm:function_approximation_KS}. Since $\mcl R$ is a bounded linear operator and $\mcl B \in \mbf X$, we have that the conditions of Thm.~\ref{thm:function_approximation_KS} are satisfied. Thus, \vspace{-1.5mm}
 \begin{align*}
&\|(I - (s-\sigma)\mcl R)^{-1} \mcl B - \mbf V (I - (s-\sigma) H)^{-1} \mbf V^* \mcl B\|_{\mbf X} \\
&\qquad\qquad \leq K\|\mcl B\|_{\mbf X} \min_{p \in \C_r[z]} \max_{z \in \Sigma} \bigg| \frac{1}{1 - (s-\sigma) z} - p (z)\bigg|\\[-7mm]
\end{align*}
for all $s \in \{s \in \C\;| \; (s-\sigma)^{-1} \notin \Sigma\}$.

Next, since the conditions of Lem.~\ref{lem:transformed_PIE_equivalence} are satisfied and $\mcl C$ is a bounded linear operator and $\mcl A - \sigma \mcl T$ has a bounded inverse, we conclude that \vspace{-1.5mm}
     \begin{align*}
     &|\hat G_r(s) - \hat G(s)| \\[-3mm]
     &\leq K \|\mcl C(\mcl A - \sigma \mcl T)^{-1} \| \|\mcl B\|\hspace{-2mm}\min_{p \in \C_r[z]} \max_{z \in \Sigma} \big| \frac{1}{1 - (s-\sigma) z} - p (z)\big|. \\[-7mm]
     \end{align*}
for all $s \in \{s \in \C\;| \; (s-\sigma)^{-1} \notin \Sigma\}$.
\end{proof}\vspace{-1.5mm}

Thm.~\ref{thm:transfer_function_error_bounds_without_stability} shows that the controllability subspace allows us to find a rational approximation of transfer functions. However, the convergence of the approximation is guaranteed only on the set $S_\Sigma := \{s \in \C\;| \; (s-\sigma)^{-1} \notin \Sigma\}$, where $\Sigma$ is defined as in~\eqref{eqn:Sigma_set}. In the next section, we use sufficient conditions for well-posedness and exponential stability of the PIE system to  characterize the set $S_\Sigma$ and to derive finite-frequency $H_\infty$ error bounds. %

 \section{Convergence of the Rational Approximation for Well-Posed and Stable PIEs}\label{sec:error_bounds}

In this section, we derive sufficient conditions for extending the error bounds in Thm.~\ref{thm:transfer_function_error_bounds_without_stability} to $s \in \{s \in \C\; |\; Re(s) > - \omega\}$. We accomplish this task in two steps. First, in Subsec.~\ref{subsec:charaterization_sigma} we use dissipativity conditions on the operator $\mcl R$ to characterize the set $\Sigma$ in~\eqref{eqn:Sigma_set}. The dissipativity condition on the operator $\mcl R$ is closely related to exponential stability and well-posedness of the PIE system -- See~\cite{jagt2026verifying}. Next, in Subsec.~\ref{subsec:convergence_exp_stable_and_well_posed}, we extend the convergence results of Thm.~\ref{thm:transfer_function_error_bounds_without_stability} and  derive the $H_\infty$ error bounds.

\subsection{Dissipativity Conditions for Characterization of $\Sigma$}\label{subsec:charaterization_sigma}
Let us find bounds on the set $\Sigma$, defined as  in~\eqref{eqn:Sigma_set}. For this task, let us consider a slight modification of strictly dissipative operators. Specifically, we assume that $\mcl R$ satisfies\vspace{-1.5mm}
\[
\ip{\mbf u}{\mcl R \mbf u}_{\mbf X} + 
\ip{\mcl R \mbf u}{\mbf u}_{\mbf X} \leq - 2\omega\ip{\mcl R \mbf u}{\mcl R \mbf u}_{\mbf X} \vspace{-1.5mm}
\]
for all $\mbf u \in \mbf X$.

This condition allows us to show that the set $\Sigma$ in~\eqref{eqn:Sigma_set} is in the left half-plane. 

\begin{lem}\label{lem:numerical_range}
    Let $\mbf X$ be a complex Hilbert space and $\mcl R:\mbf X \rightarrow\mbf X$ be a bounded linear operator. Suppose that there exists $\omega > 0$ such that \vspace{-1.5mm}
    \begin{equation}\label{eqn:LMI_lemma_numericalrange}
    \ip{\mbf u}{\mcl R \mbf u}_{\mbf X} + 
    \ip{\mcl R \mbf u}{\mbf u}_{\mbf X} \leq - 2\omega\ip{\mcl R \mbf u}{\mcl R \mbf u}_{\mbf X} \vspace{-1.5mm}
    \end{equation}
    for all $u \in \mbf X$. 
    Then, for $\Sigma$ as in~\eqref{eqn:Sigma_set} we have \vspace{-1.5mm}
    \[
    \Sigma \subseteq \bigg\{ z\in\C\;\bigg|\; \bigg|z + \frac{1}{2\omega}\bigg| \leq \frac{1}{2\omega}\bigg\}\vspace{-1mm}
    \]
\end{lem}
\begin{proof}
    Recall that $\Sigma = cl(\{\ip{\mbf x}{\mcl R \mbf x}_{\mbf X} \; | \; \mbf x \in \mbf X ,\|\mbf x\|_{\mbf X} = 1\})$. We will show for every $z \in \Sigma$, we have $|z+(2\omega)^{-1}| \leq (2\omega)^{-1}$.
    
    Suppose $z \in \{\ip{\mbf x}{\mcl R \mbf x}_{\mbf X} \; | \; \mbf x \in \mbf X\; ,\|\mbf x\|_{\mbf X} = 1\}$, then there exists $\mbf x \in \mbf X$, $\|\mbf x\|_{\mbf X} = 1$, such that $z = \ip{\mbf x}{\mcl R \mbf  x}_{\mbf X}$. Then, since $\|\mbf x\|_{\mbf X} = 1$, we have \vspace{-1.5mm}
    \[
    |z + (2\omega)^{-1}| = |\ip{\mbf x}{ (\mcl R + (2\omega)^{-1} I) \mbf x}_{\mbf X}| \leq \|\mcl R + (2\omega)^{-1} I\|_{\mbf X}. \vspace{-1.5mm}
    \]
    Then, for all $\mbf u \in \mbf X$ we have\vspace{-1.5mm}
    \begin{align*}
    &\big<(\mcl R + (2\omega)^{-1}I)\mbf u,(\mcl R + (2\omega)^{-1}I)\mbf u\big>_{\mbf X}=\ip{\mcl R \mbf u}{\mcl R \mbf u}_{\mbf X}   \\[-1.5mm]
    &\quad+ (2\omega)^{-1}(\ip{\mbf u}{\mcl R \mbf u}_{\mbf X}+ 
    \ip{\mcl R \mbf u}{\mbf u}_{\mbf X}) + (2\omega)^{-2} \ip{\mbf u}{\mbf u}_{\mbf X}.\\[-7mm]
    \end{align*}
    Finally, by~\eqref{eqn:LMI_lemma_numericalrange}, we have \vspace{-1.5mm}
        \[
    \big<\big(\mcl R + (2\omega)^{-1} I\big)\mbf u,\big(\mcl R + (2\omega)^{-1} I\big)\mbf u\big>_{\mbf X} \leq (2\omega)^{-2} \ip{\mbf u}{\mbf u}_{\mbf X}.\vspace{-1.5mm}
    \]
    Thus, $\|\mcl R + \frac{1}{2\omega} I\|_{\mbf X}\leq (2\omega)^{-1}$ and $|z + (2\omega)^{-1}| \leq  (2\omega)^{-1}$. \vspace{1mm}

    Let $\Theta:= \big\{ z\in\C\;\big|\; \big|z +  \frac{1}{2\omega}\big| \leq  \frac{1}{2\omega}\big\}$. Therefore, we conclude that $\big\{\ip{\mbf x}{\mcl R \mbf x}_{\mbf X} \; \big| \; \mbf x \in \mbf X\; ,\|\mbf x\|_{\mbf X} = 1\big\} \subseteq \Theta$. 
    Finally, since $\Theta$ is closed, we conclude that $\Sigma \subseteq \Theta$.
\end{proof}\vspace{-1.5mm}

The condition in~\eqref{eqn:LMI_lemma_numericalrange} with a specific choice of the Hilbert space $\mbf X$ and the operator $\mcl R$ is used for PIE systems to verify exponential stability of the system. In the next subsection, we specify the space $\mbf X$ and generalize the convergence results in Thm.~\ref{thm:transfer_function_error_bounds_without_stability} for $s \in \C$, $Re(s) > -\omega$.

\subsection{Frequency Domain Error Bounds for Rational Approximation of Transfer Function}
\label{subsec:convergence_exp_stable_and_well_posed}

As discussed in~\cite{shivakumar2022dual}, the PIE system in~\eqref{eqn:PIE} is exponentially stable with decay rate $\omega > 0$ if there exists $\mcl P \geq \eta I$ such that \vspace{-4mm}
\begin{equation}\label{eqn:LMI_exp_stability_not_lemma}
\mcl A^* \mcl P \mcl T + \mcl T^*\mcl P \mcl A \leq - 2\omega \mcl T^* \mcl P \mcl T.\vspace{-1.5mm}
\end{equation}

An extension of this result in~\cite{jagt2026verifying} shows that the well-posedness of the PIE system can be verified using the LMI in~\eqref{eqn:LMI_exp_stability_not_lemma}, assuming that there exists $\sigma \in \R$ such that $\mcl A - \sigma \mcl T $ is invertible. We have already used this assumption in Lem.~\ref{lem:transformed_PIE_equivalence} and Thm.~\ref{thm:transfer_function_error_bounds_without_stability}. Furthermore, the exponential stability and the well-posedness of the PIE in~\eqref{eqn:PIE} allow us to find a space $\mbf X$, in which the operator $\mcl R$ is strictly dissipative.

\begin{lem}\label{lem:new_space}
    Let $\mbf Z:= \C^n \times L_2([a, b],\C^m)$, $\mcl A, \mcl T :\mbf Z \rightarrow \mbf Z$ be PI operators. Suppose there exist $\omega, \eta > 0$ and a PI operator $\mcl P = \mcl P^*  \geq \eta I$ such that \vspace{-1.5mm}
    \begin{equation}\label{eqn:LMI_lemma_space}
    \mcl A^* \mcl P \mcl  T + \mcl T^* \mcl P \mcl A \leq - 2\omega\mcl T^* \mcl P \mcl T.\vspace{-1.5mm}
    \end{equation}
    Suppose there exists $\sigma > -\omega$ such that $\mcl A - \sigma \mcl T $ has a bounded inverse.
    Let $\mcl R:= \mcl T(\mcl A - \sigma \mcl T)^{-1}$. Define the Hilbert space $\mbf X$ to be $\C^n\times L_2([a, b],\C^m)$ with the inner product $\ip{\mbf u}{\mbf v}_{\mbf X} = \ip{\mbf u}{\mcl P\mbf v}_{\mbf Z}$. Then, \vspace{-1.5mm}
\[
\ip{\mbf u}{\mcl R \mbf u}_{\mbf X} + 
\ip{\mcl R \mbf u}{\mbf u}_{\mbf X} \leq - 2(\omega + \sigma)\ip{\mcl R \mbf u}{\mcl R \mbf u}_{\mbf X} \vspace{-1.5mm}
\]
for all $\mbf u\in \mbf X$.
\end{lem}\vspace{-1.5mm}
\begin{proof}  
    By adding $-2\sigma\mcl T^*\mcl P \mcl T$ to the LMI in~\eqref{eqn:LMI_lemma_space}, we have\vspace{-1.5mm}
\[
\mcl T^* \mcl P (\mcl A - \sigma\mcl T) + (\mcl A - \sigma\mcl T) ^* \mcl P \mcl T \leq -2(\omega + \sigma) \mcl T^* \mcl P \vspace{-1.5mm}\mcl T.
\]
Next, since $ (\mcl A - \sigma\mcl T)$ is invertible, for $\mcl R= \mcl T(\mcl A - \sigma\mcl T)^{-1}$ we have $\mcl R^* \mcl P + \mcl P \mcl R \leq -2(\omega + \sigma)\mcl R^*\mcl P\mcl R.$
Thus, for all $\mbf u \in \mbf X $, we have \vspace{-1.5mm}
\[
    \ip{\mbf u}{\mcl R \mbf u}_{\mbf X} + 
    \ip{\mcl R \mbf u}{\mbf u}_{\mbf X} \leq - 2(\omega + \sigma)\ip{\mcl R \mbf u}{\mcl R \mbf u}_{\mbf X}.\vspace{-7mm}
\]
\end{proof}\vspace{-1.5mm}

Applying Thm.~\ref{thm:transfer_function_error_bounds_without_stability} with the space $\mbf X$ as in Lem.~\ref{lem:new_space} allows us to extend the convergence results of the transfer function approximation.
 
\begin{thm}\label{thm:transfer_function_error_bounds}
     Let $\mbf Z:= \C^n \times L_2([a, b], \C^m)$, $\mcl D\in\C$, $\mcl B \in \mbf Z$ and $\mcl T, \mcl A:\mbf Z \rightarrow \mbf Z$, $\mcl C:\mbf Z \rightarrow \C$ be PI operators. Suppose there exist $\eta, \omega  >0$ and a PI operator $\mcl P = \mcl P^* \geq \eta I$ such that 
     \vspace{-1.5mm}
     \begin{equation}\label{eqn:LMI_thm_final}
     \mcl T^* \mcl P \mcl A + \mcl A^* \mcl P \mcl T \leq -2\omega \mcl T^* \mcl P \mcl T.\vspace{-2.5mm}
     \end{equation}  
     Suppose there exists $\sigma > -\omega$ such that $\mcl A - \sigma \mcl T$ has a bounded inverse. Define $\mcl R\hspace{-0.5mm} := \hspace{-0.5mm}\mcl T(\mcl A - \sigma \mcl T)^{-1}\hspace{-1mm}$.     
    
     Suppose there exist $\{\mbf q_i\}_{i = 1}^r\hspace{-1mm}\subset\hspace{-1mm}  \texttt{span}\{\mcl B, \dots ,\mcl R^{r-1}\mcl B\}$ such that $\ip{\mbf q_i}{\mcl P\mbf q_j}_{\mbf Z} = \delta_{ij}$. 
     Define $H_{ij} = \ip{\mbf q_i}{\mcl P\mcl R \mbf q_j}_{\mbf Z}$,
     \begin{align}\label{eqn:transfer_function_approximation_thm}
         &\hat G(s) = \mcl C(s\mcl T - \mcl A)^{-1} \mcl B + \mcl D\\
         &\hat G_r(s) = \mcl C (\mcl A - \sigma \mcl T)^{-1} \mbf V ((s-\sigma)H - I)^{-1} \mbf V^* \mcl P\mcl B + \mcl D,\notag\\[-7mm]\notag
     \end{align}  
     where $\mbf V:= \bmat{\mbf q_1, \dots , \mbf q_r}$.
    
     Furthermore, let $\Theta \subset \C$ be defined as \vspace{-1.5mm}
          \begin{equation}\label{eqn:Theta_set}
     \Theta = \{z \in \C\;|\; |z + (2\omega + 2\sigma)^{-1}| \leq (2\omega + 2\sigma)^{-1}\}. \vspace{-1.5mm}
     \end{equation}
     Define $M:=\frac{\sqrt{\|\mcl P\|}}{\sqrt{\eta}}\|\mcl C(\mcl A - \sigma \mcl T)^{-1} \| \|\mcl B\|$. \\
     Then, there exists $K \geq0$ such that
     {\small\begin{align}\label{eqn:errorbounds_thm}
     |\hat G_r(s) - \hat G(s)| \leq K M \min_{p \in \C_r[z]} \max_{z \in \Theta} \big| \frac{1}{1 \hspace{-0.5mm}-\hspace{-0.5mm} (s-\sigma) z} - p (z)\big|\\[-7mm]\notag
     \end{align}}
    for all $s \in  \{ s\in \C\;|\; Re(s) > -\omega\}$.
\end{thm}\vspace{-3mm}
\begin{proof} 
In the proof, we verify the conditions of Thm.~\ref{thm:transfer_function_error_bounds_without_stability} to show the error bounds of the transfer function for the space $\mbf X$ as defined in Lem.~\ref{lem:new_space}. Recall, $\mbf X$ is $\C^n \times L_2([a, b], \C^m)$ with the inner product defined as $\ip{\mbf u}{\mbf v}_{\mbf X} = \ip{\mbf u}{\mcl P \mbf v}_{\mbf Z}$. 

Our proof consists of three parts. First, we show that for the space $\mbf X$, the approximation in~\eqref{eqn:transfer_function_approximation_PIE_transformed} has the form as in Eqn.~\eqref{eqn:transfer_function_approximation_thm}. Next, we show that for $\Sigma$ as defined in~\eqref{eqn:Sigma_set}, we have $\Sigma \subseteq \Theta$. Finally, we show that if $Re(s) > -\omega$, then $(s - \sigma)^{-1} \notin \Theta$. Thus, using error bounds as in Thm.~\ref{thm:transfer_function_error_bounds_without_stability}, we conclude the theorem statement.

Let us start with showing that the approximation in~\eqref{eqn:transfer_function_approximation_PIE_transformed} has the form as in~\eqref{eqn:transfer_function_approximation_thm}. Let $\mbf V^*_\mbf X$ be an adjoint of $\mbf V$ for the inner product $\ip{}{}_{\mbf X}$. Since $\mcl P =\mcl P^*$  we have\vspace{-1.5mm}
\begin{equation*}
    \ip{\mbf u}{\mbf V \alpha}_{\mbf X} = \ip{\mbf u}{\mcl P\mbf V \alpha}_{\mbf Z} = (\mbf V^* \mcl P\mbf u )^\dagger \alpha \vspace{-1.5mm}
\end{equation*}
for all $\mbf u \in \mbf X$ and $\alpha \in \C^r$.

Thus, $\mbf V_{\mbf X}^* = \mbf V^* \mcl P$ and for $\hat G_r(s)$ as in~\eqref{eqn:transfer_function_approximation_thm} we have \vspace{-1.5mm}
\begin{equation*}
\hat G_r(s) =  \mcl C (\mcl A - \sigma \mcl T)^{-1} \mbf V ((s-\sigma)H - I)^{-1} \mbf V^*_X\mcl B + \mcl D, \vspace{-1.5mm}
\end{equation*}
which coincides with the approximation in Thm.~\ref{thm:transfer_function_error_bounds_without_stability}.

Furthermore, since  $\sqrt{\eta}\|\mbf u\|_{\mbf Z} \leq \|\mbf u\|_{\mbf X} \leq \sqrt{\|\mcl P\|} \|\mbf u\|_{\mbf Z}$ for all $\mbf u\in \mbf Z$, we have $\| \mcl C (\mcl A - \sigma \mcl T)^{-1}\|_{\mbf X\rightarrow\C} \leq \frac{1}{\sqrt{\eta}} \| \mcl C (\mcl A - \sigma \mcl T)^{-1}\|_{\mbf Z\rightarrow\C}$ and $\|\mcl B\|_{\mbf  X} \leq \sqrt{\|\mcl P\|}\|\mcl B\|_{\mbf Z}$.

Then, by Thm.~\ref{thm:transfer_function_error_bounds_without_stability}, we conclude that 
there exists $K \geq 0$ such that \vspace{-2.5mm}{\small
 \begin{align*}
     &|\hat G_r(s) - \hat G(s)| \leq K M \hspace{-1.5mm}\min_{p \in \C_r[z]} \max_{z \in \Theta} \big| \frac{1}{1 \hspace{-0.5mm}-\hspace{-0.5mm} (s\hspace{-0.5mm}-\hspace{-0.5mm}\sigma)\hspace{-0.5mm} z} \hspace{-0.5mm}-\hspace{-0.5mm} p (z)\big|\notag\\[-7mm]\notag
     \end{align*}}
     for all $s \in \{s \in \C\;| \; (s-\sigma)^{-1} \notin \Sigma\}$
     
Let us show that $\Sigma \subseteq \Theta$. For this task, we use Lemmas~\ref{lem:numerical_range} and~\ref{lem:new_space}. Since the conditions of Lem.~\ref{lem:new_space} are satisfied, for all $\mbf u \in \mbf X$ we have $\ip{\mbf u}{\mcl R \mbf u}_{\mbf X} + 
    \ip{\mcl R \mbf u}{\mbf u}_{\mbf X} \leq - 2(\omega + \sigma)\ip{\mcl R \mbf u}{\mcl R \mbf u}_{\mbf X}$. 
    
Thus, by Lem.~\ref{lem:numerical_range}, we conclude that\vspace{-1.5mm}
\[
\Sigma \subseteq \Theta = \{z \in \C\;|\; |z + (2\omega + 2\sigma)^{-1}| \leq (2\omega + 2\sigma)^{-1}\}. \vspace{-1.5mm}
\]
Now, let us show that if $Re(s) > -\omega$ and $s \neq \sigma$, then $(s-\sigma)^{-1} \notin \Theta$. For $s = \sigma$, the error bounds trivially hold. Let $z = s-\sigma$. Since $Re(z) > -(\omega+\sigma)$, we have $(Re(z) + 2(\omega+\sigma))^2 > Re(z)^2$. Thus, if $Re(z)> -(\omega+\sigma)$, \vspace{-1.5mm}
\[
|z + 2(\omega+\sigma)| >|z|.\vspace{-1.5mm}
\]
Next, dividing by $|z|$ and $2(\omega + \sigma)$, we obtain\vspace{-1.5mm}
\[
|z^{-1} + (2\omega+2\sigma)^{-1}| > (2\omega+2\sigma)^{-1}\vspace{-1.5mm}
\]
Thus, if $Re(z) > -(\omega+\sigma)$ and $z \neq 0$, then $z^{-1} \notin \Theta$. Furthermore, since $s = z+\sigma$, if $Re(s) > -\omega$ and $s \neq \sigma$, we have $(s - \sigma)^{-1} \notin \Theta$.

Finally, by Thm.~\ref{thm:transfer_function_error_bounds_without_stability} we conclude that \vspace{-1.5mm}
     \begin{align*}
     |\hat G_r(s) - \hat G(s)| \leq KM\min_{p \in \C_r[z]} \max_{z \in \Theta} \big| \frac{1}{1 \hspace{-0.5mm}-\hspace{-0.5mm} (s\hspace{-0.5mm}-\hspace{-0.5mm}\sigma)\hspace{-0.5mm} z} \hspace{-0.5mm}-\hspace{-0.5mm} p (z)\big|\notag\\[-7mm]\notag
     \end{align*}
for all $s\in \{s \in \C\; |\; Re(s) > - \omega\}$.
\end{proof}\vspace{-1.5mm}

Thm.~\ref{thm:transfer_function_error_bounds} shows sufficient conditions for convergence of the transfer function approximation. The constant in Thms.~\ref{thm:function_approximation_KS},\ref{thm:transfer_function_error_bounds_without_stability} and~\ref{thm:transfer_function_error_bounds} can always be chosen such that $K \leq 5$~\cite{guttel2010rational, crouzeix2017numerical}.

For finite frequency, we have constructive error bounds.
\begin{cor}\label{cor:finite_Hinf_bounds}
     Suppose that the conditions of Thm.~\ref{thm:transfer_function_error_bounds} are satisfied. Then, for all $\Omega > 0$ \vspace{-1.5mm}
     \begin{equation*}
              \sup_{\nu \in [-\Omega, \Omega]}|\hat G_r(j
              \nu) - \hat G(j
              \nu)| \leq K M \frac{2(\omega + \sigma)}{2\omega + \sigma}\frac{\delta^r}{1 - \delta},\vspace{-1.5mm}
     \end{equation*}
     where  $\delta \hspace{-0.5mm}=\hspace{-0.5mm} \sqrt{\hspace{-0.5mm}\frac{\sigma^2\hspace{-0.1mm} +\hspace{-0.1mm} \Omega^2}{(2\omega\hspace{-0.1mm} +\hspace{-0.1mm} \sigma)^2\hspace{-0.1mm} +\hspace{-0.1mm} \Omega^2}}$.
\end{cor}
\begin{proof}
    In this proof, we will show that
    \[
   \hspace{-2mm} \sup_{\nu \in [-\Omega, \Omega]}\min_{p \in \C_r[z]} \max_{z \in \Theta} \big| \frac{1}{1 \hspace{-0.5mm}-\hspace{-0.5mm} (j\nu\hspace{-0.5mm}-\hspace{-0.5mm}\sigma)\hspace{-0.5mm} z} - p (z)\big| \leq \frac{2(\omega + \sigma)}{2\omega + \sigma}\frac{\delta^r}{1 - \delta}.\vspace{-1.5mm}
    \]
    Let $\hat z = 2(\omega + \sigma) z + 1$, then $|\hat z| \le 1$ and \vspace{-1.5mm}
{\small\[
    \frac{1}{1 \hspace{-0.5mm}-\hspace{-0.5mm} (j\nu\hspace{-0.5mm}-\hspace{-0.5mm}\sigma)z} = \frac{2(\omega + \sigma)}{2\omega + \sigma  +  j\nu} \frac{1}{1 - \frac{j\nu - \sigma}{2\omega + \sigma  +  j\nu} \hat z}.\vspace{-1.5mm}
    \]}
    Let $\rho(\nu) = \frac{j\nu - \sigma}{2\omega + \sigma  +  j\nu}$. Then, since $\omega > 0$ and $ \omega + \sigma > 0$, for all $\nu \in \R$, we have $|\rho(\nu)| < 1$.

    Let $p_0(\hat z):= \sum_{i = 0}^{r- 1} \rho(\nu)^i \hat z^i$. Then, since $|\rho(\nu) \hat z| < 1$, using the Taylor expansion, we have for all $\nu \in \R$ \vspace{-1.5mm}
    {\small\[
    \max_{|\hat z| \le 1} \big|\frac{1}{1 - \rho(\nu) \hat z} - p_0(\hat z)\big| \leq \frac{|\rho(\nu)|^r}{1 - |\rho(\nu)|}.\vspace{-1.5mm}
    \]}
    Finally, let $p_\nu(z) = \frac{2(\omega + \sigma)}{2\omega + \sigma  +  j\nu} p_0(2(\omega + \sigma) z + 1)$. Then, since $\sup_{\nu \in [-\Omega, \Omega]} |\rho(\nu)| = |\frac{j\Omega - \sigma}{2\omega + \sigma  +  j\Omega}|$,  we observe that \vspace{-1.5mm}
    {\small\[
    \sup_{\nu \in [-\Omega, \Omega]}\min_{p \in \C_r[z]} \max_{z \in \Theta} \big| \frac{1}{1 \hspace{-0.5mm}-\hspace{-0.5mm} (j\nu\hspace{-0.5mm}-\hspace{-0.5mm}\sigma)\hspace{-0.5mm} z} - p_\nu (z)\big| \leq \frac{2(\omega + \sigma)}{2\omega + \sigma}\frac{\delta^r}{1 - \delta},\vspace{-1.5mm}
    \]}
    which concludes the corollary.
\end{proof}

\subsection{State-Space Representation of the Approximation}

For practical usage of the rational approximation in~\eqref{eqn:transfer_function_approximation_PIE_transformed}, a state-space representation may be required.\vspace{-1.5mm}
\begin{align} \label{eqn:ODE_approximation}
    H \dot x(t) &= (I + \sigma H) x(t) + B_a u(t) \\
    y(t) & = C_a x(t) + D_a u(t),\notag\\[-7mm]\notag
\end{align}
where $x(t) \in \C^r$ is the ODE state, $u(t) \in \C$ is the input, $y(t) \in \C$ is the output and $C_a: = \mcl C (\mcl A -\sigma \mcl T)^{-1}\mbf V \in \C^{1 \times r}$, $B_a := \mbf V^* \mcl P\mcl B  \in \C^r$ and $D_a :=\mcl D$ ($\mbf V$, $H$ and $\mcl P$ as defined in Thm.~\ref{thm:transfer_function_error_bounds}). 

The exponential stability of the reduced-order system in~\eqref{eqn:ODE_approximation} is established as follows.
\begin{cor}\label{cor:approximation_is_stable}
     Given bounded linear operators $\mcl T, \mcl A, \mcl B, \mcl C, \mcl D$, suppose that the conditions of Thm.~\ref{thm:transfer_function_error_bounds} are satisfied. Let $\mcl T$ be injective and $H = \mbf V^* \mcl P \mcl R \mbf V$. Then $H$ is invertible, and the system in~\eqref{eqn:ODE_approximation} is exponentially stable.
\end{cor}
\begin{proof}
    Let $\mbf X$, $H$ and $\mbf V^*_{\mbf X}$ be defined as in the proof of Thm.~\ref{thm:transfer_function_error_bounds}. Recall that $\mbf V^*_{\mbf X} = \mbf V^*\mcl P $ and $H = \mbf V^* \mcl P \mcl R \mbf V = \mbf V^*_{\mbf X} \mcl R \mbf V$.
    
    Then, as shown in the proof of Thm.~\ref{thm:transfer_function_error_bounds} , we have\vspace{-1.5mm}
    \[
    \ip{\mbf u}{\mcl R \mbf u}_{\mbf X} + \ip{\mcl R\mbf u}{\mbf u}_{\mbf X} \leq - 2(\omega+\sigma)\ip{\mcl R \mbf u}{\mcl R \mbf u}_{\mbf X}\vspace{-1.5mm}
    \]
    for all $\mbf u \in \mbf X$. Thus, for any $q \in \C^r$ we have\vspace{-1.5mm}
    \[
    q^\dagger\big(\mbf V^*_{\mbf X} \mcl R \mbf V\hspace{-0.8mm}+\hspace{-0.8mm} (\mbf V^*_{\mbf X} \mcl R \mbf V)^\dagger\big) q \hspace{-0.8mm}\leq\hspace{-0.8mm} -2(\omega+\sigma)  q^\dagger \ip{\mcl R \mbf V}{\mcl R \mbf V}_{\mbf X}q.
    \]
    Next, since $\mbf V\mbf V^*_{\mbf X} $ is a projection and $H = \mbf V^*_{\mbf X} \mcl R \mbf V$, we have that $\|H q\|^2 = \|\mbf V^*_{\mbf X} \mcl R \mbf V q\|^2 \leq \|\mcl R \mbf V q\|^2$. 

    Therefore, $-\ip{\mcl R \mbf V}{\mcl R \mbf V}_{\mbf X} \leq -H^\dagger H$ and \vspace{-1.5mm}
    \begin{equation}\label{eqn:LMI_H}
    H^\dagger + H \leq -2(\omega + \sigma)H^\dagger H.\vspace{-1.5mm}
    \end{equation}
    Next, since $\mcl T$ is injective, we have that $\mcl R$ is injective.
    Thus, $\ip{\mcl R \mbf Vq}{\mcl R \mbf Vq}_{\mbf X} > 0$ for all $q \neq 0$, which implies $q^\dagger(H^\dagger + H)q < 0$ for all $q \neq 0$. Thus, $H$ is invertible.

    To show that the system in~\eqref{eqn:ODE_approximation} is stable, we define the state-space matrix $A_a := H^{-1} + \sigma I$. Then, Eq.~\eqref{eqn:LMI_H} implies that $A_a^\dagger + A_a \leq - 2\omega I$,
    which implies exponential stability of the reduced-order system in~\eqref{eqn:ODE_approximation}.
\end{proof}

\section{Application to PDE transfer functions}\label{sec:PDE_transfer_function}

In the previous sections, we considered the input-output map modeled by the PIE system in~\eqref{eqn:PIE}. In this section, we apply the results in Sec.~\ref{sec:error_bounds} to the PDE transfer functions. Specifically, let us consider an input-output map modeled by a PDE of the form\vspace{-1.5mm}
\begin{align}\label{eqn:PDE}
    \dot {\mbf v}(t) & = A \mbf v(t) + Bu(t) \qquad \mbf v(0) = 0 \\
    y(t) &= C\mbf v(t) + D u(t),\notag\\[-7mm]\notag
\end{align}
where $u(t) \in \C$, $y(t) \in \C$ are the input and output of the system. $\mbf v(t) \in Dom(A)$ is the PDE state. $A:Dom(A)\rightarrow L_2([a, b], \C^d)$ is a linear differential operator, $B:\C \rightarrow L_2([a, b], \C^d)$, $C:Dom(A) \rightarrow \C$ are bounded linear operators, $D \in \C$ and $Dom(A)$ is defined as\vspace{-1.5mm}
\begin{align}\label{eqn:BCs}
Dom(A) \hspace{-0.5mm}=\hspace{-0.5mm} \bigg\{\hspace{-0.5mm} \mbf v \hspace{-0.5mm}\in \hspace{-0.5mm}L_2([a, b], \C^d)\big|& \partial_\theta^{n} \hspace{-0.5mm}\mbf v\hspace{-0.5mm}\in \hspace{-0.5mm}L_2([a, b], \C^d), \hspace{-3mm}\\[-2mm]
&\hspace{-22mm}{\scriptsize M_1 \bmat{\mbf v(t, a) \\ \vdots \\ \mbf  \partial_{\theta}^{n-1}\mbf v(t, a)} + M_2 \bmat{\mbf v(t, b) \\ \vdots \\ \mbf  \partial_{\theta}^{n-1}\mbf v(t, b)} = 0}\bigg\}, \notag\\[-6mm]\notag
\end{align}
 where $M_1, M_2 \in \C^{dn\times dn}$.
 
 Operators $A, B, C$ have the form \vspace{-1.5mm}
{\small\begin{align}\label{eqn:operators_ABC}
(A \mbf v)(\theta) &= \sum_{i = 0}^{n} A_i(\theta) \partial_\theta^i \mbf v(\theta), \qquad \theta \in [a, b]\\[-1mm]
C \mbf v &= \int_a^b c(\theta) \partial_\theta^{n}\mbf v(\theta) d\theta, \qquad
(Bu)(\theta) = b(\theta) u.\notag\\[-7mm]\notag
\end{align}}

Let the $C_0$-semigroup generated by operator $A$ be exponentially stable, and let $\omega$ be an exponential decay rate of the semigroup. Then the transfer function of the PDE~\eqref{eqn:PDE} is defined as \vspace{-2.5mm}
\begin{equation}\label{eqn:transferfunction_PDE}
    \hat G(s) = C(s I - A)^{-1} B + D,\vspace{-1.5mm}
\end{equation}
for all $s \in \{ s \in \C\; | \; Re(s) >  -\omega \}$ (Lem 4.3.6 in~\cite{curtain2020introduction}).

\subsection{Equivalence between PDE and PIE transfer functions}\label{subsec:PIE_transfer_function}

As discussed in~\cite{peet2021partial}, the equivalence between the PDE solution and the PIE solution is established by introducing an invertible map $\mcl T:L_2[a,b]\rightarrow Dom(A)$, where $Dom(A)$ is as in Eqn.~\eqref{eqn:BCs}. The operator $\mcl T$ is defined as the inverse of the differential operator $\partial_\theta^n :Dom(A)\rightarrow L_2[a, b]$. For suitable boundary conditions, the operator $\mcl T$ is a PI operator~\cite{peet2021partial}. Furthermore, let operator $\mcl A, \mcl B, \mcl C, \mcl D$ in~\eqref{eqn:PIE} be given as\vspace{-1.5mm}
\begin{align}\label{eqn:PDE_to_PI_operators}
    \mcl A:= A \circ \mcl T, \quad \mcl C:= C\circ\mcl T, \quad \mcl B:= B, \quad \mcl D = D.\\[-8mm]\notag
\end{align}
Then $\mcl A, \mcl B, \mcl C, \mcl D$ are also PI operators, and the states of the PIE in Eqn.~\eqref{eqn:PIE} represent the PDE states in~\eqref{eqn:PDE} as  \vspace{-1.5mm}
\[
\mbf x(t, \theta) = \partial_\theta^n \mbf v(t, \theta) \qquad \text{and} \qquad\mbf v(t) = \mcl T \mbf x(t).\vspace{-1.5mm}
\]

In the next theorem, we establish the equivalence between PDE and PIE transfer functions in~\eqref{eqn:transferfunction_PDE} and ~\eqref{eqn:transferfunction_PIE}.
 
\begin{thm}\label{thm:transfer_functions_equivalence}
   Let $A:Dom(A)\rightarrow  L_2([a,b], \C^n)$ be an infinitesimal generator of a $C_0$-semigroup with growth bound $\alpha$. Suppose that $B:\C \rightarrow  L_2([a,b], \C^n)$, $C: Dom(A) \rightarrow \C$ are bounded linear operators and $D\in \C$.
   Suppose that $\mcl  T:L_2([a, b], \C^n)\rightarrow L_2([a, b], \C^n)$ is an injective bounded linear operator with $Range(\mcl T) = Dom(A)$.
   Define $\mcl A, \mcl B, \mcl C,\mcl D$ as in Eqn.~\eqref{eqn:PDE_to_PI_operators}. Then,\vspace{-1.5mm}
   \[
   C(sI - A)^{-1}B + D= \mcl C(s\mcl T - \mcl A)^{-1}\mcl B + \mcl D\vspace{-1.5mm}
   \]
   for all $s \in \{ s\in \C\;| \; Re(s) > \alpha\}$.
\end{thm}\vspace{-1.5mm}
\begin{proof}
    For the proof, we show that\vspace{-1.5mm}
\[
(sI - A)^{-1} \mbf u = \mcl T (s \mcl T -  A  \mcl T)^{-1} \mbf u\vspace{-1.5mm}
\]
for all $s \in  \{ s\in \C\;| \; Re(s) > \alpha\}$ and $\mbf u \in L_2([a, b], \C^n)$.

Then, using operators as in~\eqref{eqn:PDE_to_PI_operators}, we conclude the theorem.

Since $A$ is an infinitesimal generator, we have $A$ is closed and $Dom(A)$ is dense in $L_2([a, b], \C^n)$ (Thm.2.1.10 in~\cite{curtain2020introduction}). Thus, by the Hille-Yosida Theorem (Thm.2.1.12 in~\cite{curtain2020introduction}), $(sI - A)$ has a bounded inverse for all $s \in  \{ s\in \C\;| \; Re(s) > \alpha\}$. 

Next, since $A$ is closed and $\mcl T$ is bounded, we have $\mcl A:= A \circ \mcl T$ is a bounded operator by the closed graph theorem. 

Now, since $\mcl T: L_2([a,b], \C^n) \rightarrow Dom(A)$ is injective and $(sI -A):Dom(A)\rightarrow L_2([a,b], \C^n)$ is bijective, we have that $s\mcl T - \mcl A:L_2([a,b], \C^n) \rightarrow L_2([a,b], \C^n)$ is a bijective bounded linear operator. Hence, by the bounded inverse theorem, $(s\mcl T - \mcl A)^{-1}$ is bounded.

Furthermore, let $\mbf v = (sI - A)^{-1} \mbf u$ for some $\mbf u \in L_2([a, b], \C^n)$ and $Re(s) > \alpha$. Then, since $A:Dom(A) \rightarrow L_2([a, b], \C^n)$, we have $\mbf v \in Dom(A)$. Next, since $\mbf v \in Dom(A)$ and $\mcl T:L_2([a, b], \C^n) \rightarrow Dom(A)$ is invertible, we have $\mbf v =  \mcl T \mcl T^{-1} \mbf v = \mcl T \mcl T^{-1} (sI - A)^{-1} \mbf u$.

Thus, since $s\mcl T - \mcl A$ is bounded and bijective, we have \vspace{-1.5mm}
\[
\mcl T(s \mcl T - \mcl A)^{-1} = \mcl T\mcl T^{-1} (s I - \mcl A)^{-1},\vspace{-1.5mm}
\]
which implies that $C(sI - A)^{-1}B = \mcl C(s\mcl T - \mcl A)^{-1}\mcl B$. \vspace{-1.5mm}
\end{proof}

Thm.~\ref{thm:transfer_functions_equivalence} shows that the PDE transfer function in~\eqref{eqn:transferfunction_PDE} and the PIE transfer function in~\eqref{eqn:transferfunction_PIE} are equivalent. 
Specifically, as discussed in~\cite{curtain2020introduction}, the PDE transfer functions have the form as in Eqn.~\eqref{eqn:transferfunction_PDE} for $Re(s) > - \omega$, where $\omega$ is an exponential decay rate of the $C_0$-semigroup generated by $A$. 
Feasibility of the LMI in~\eqref{eqn:LMI_thm_final}, together with the existence of $\sigma$ such that $\mcl A - \sigma \mcl T$ is boundedly invertible, is sufficient for $\omega$ to be the exponential decay rate, as established in Thm.~12 in~\cite{jagt2026verifying}.   Furthermore, Thm.~\ref{thm:transfer_functions_equivalence} shows that the PIE transfer functions in~\eqref{eqn:transferfunction_PIE} and the PDE transfer functions in~\eqref{eqn:transferfunction_PDE} coincide on the set $\{s \in \C\;|\; Re(s) > -\omega\}$. Thus, the error bounds in Thm.~\ref{thm:transfer_function_error_bounds} can be applied directly to rational approximations of the PDE transfer functions. Furthermore, since $\mcl T$ is injective, the conditions of Cor.~\ref{cor:approximation_is_stable} are satisfied, and the reduced-order system is exponentially stable.

\subsection{Implementation details}\label{sec:implementation}
\label{subsec:implementation}
In this subsection, we briefly discuss how to implement the proposed approach. For this task, suppose we are given a well-posed and exponentially stable PIE system as in~\eqref{eqn:PIE}. The PDE-to-PIE transformation is discussed in~\cite{peet2021partial}. Next, we discuss challenges for obtaining the approximation in ~\eqref{eqn:transfer_function_approximation_PIE_transformed}.


\textbf{Choice of $\sigma$.} Given a well-posed and exponentially stable PIE, the first step is to find $\sigma \in \R$ such that $(\mcl A - \sigma \mcl T)$ is invertible. For simple PDEs, we may derive $\sigma$ and an inverse of $\mcl A -\sigma \mcl T$ analytically.  Thm.~\ref{thm:transfer_function_error_bounds} requires $\sigma > -\omega$. In the numerical examples, we choose $\sigma=0$, which simplifies the construction. However, for general PDEs, the inversion may still require some additional approximation methods~\cite{shivakumar2022dual}.

\textbf{Computing the basis of the Krylov Subspace.} Given a well-posed and exponentially stable PIE, the next step is to find a basis of the Krylov subspace of the space $\mbf X$ as discussed in Sec.~\ref{sec:Arnoldi_iterations}. The inner product in $\mbf X$ is defined by a positive-definite operator $\mcl P$. To compute the orthonormal basis, we use infinite-dimensional Arnoldi iterations, for example as discussed in~\cite{appeltans2023structure}.

 \section{Numerical Examples}\label{sec:numerical_examples}

In the following examples, we consider the problem of rational approximation of PDE transfer functions. For all examples, we first transform the system to PIE representation. Then, we compute rational approximations as given in Thm.~\ref{thm:transfer_function_error_bounds}.

\subsection{Stability Properties of the Reduced-Order System}
In the first example, we demonstrate the stability property of the 
reduced-order system. For this task, we consider the PDE with a parameter $\alpha \in \R$ of the form \vspace{-1.5mm}
\begin{align*}
    &\partial_t \mbf v(t, \theta) = \alpha\mbf v(t, \theta) - \partial_\theta \mbf v(t, \theta) + u(t) \\[-1mm]
    &\mbf v(t, 0)= 0,\qquad y(t) =  \mbf v(t, 1),\\[-7mm] 
\end{align*}

Using the Krylov subspace method,  Cor.~\ref{cor:approximation_is_stable} guarantees the exponential stability of the reduced-order model. Specifically, for any $\alpha \in \R$ we obtain a stable state matrix. For $\alpha = 1$, we have $A_r = \bmat{-1.14 &  1.47\\
   -4.87 &  -5.18}$ with eigenvalues $\lambda(A_r) = -3.16 \pm 1.76j$.

However, the Galerkin projection does not guarantee the stability of the reduced model. Specifically,  let us consider basis functions $\{\sin(\pi k \theta)\}_{k = 1}^\infty$ in $L_2([0, 1])$. The projection of the differential operator $A:= \alpha - \partial_\theta$ onto $\{\sin(\pi \theta), \sin(2\pi \theta)\}$ gives us $A_G = \bmat{\alpha & \frac{8}{3}\\ -\frac{8}{3} & \alpha}$,
which has eigenvalues $\lambda = \alpha \pm \frac{8}{3}j$. Thus, for $\alpha \geq 0$, the reduced-order system is not exponentially stable.

\begin{figure}
    \centering\vspace{2mm}
    \includegraphics[width=0.7\linewidth]{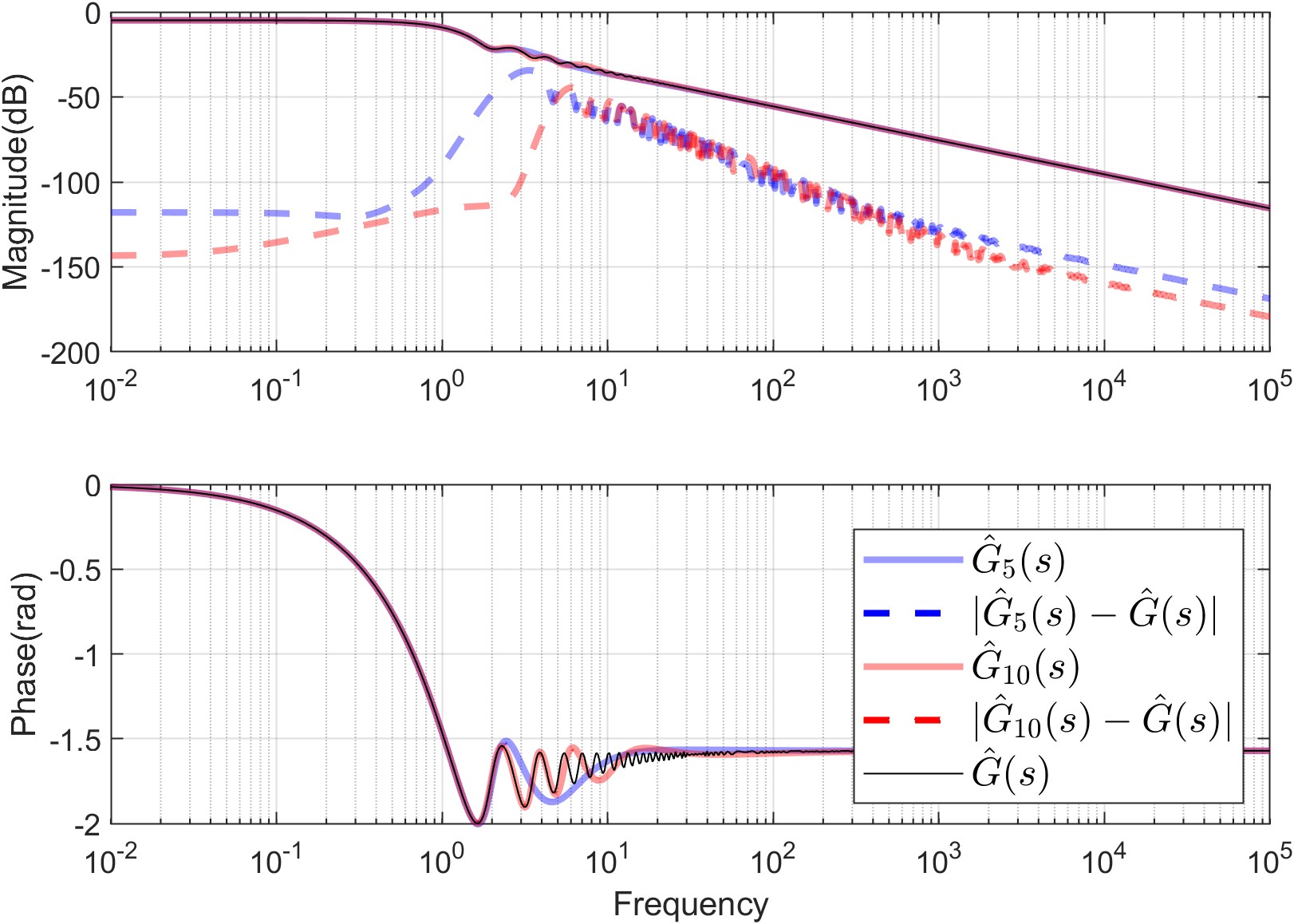}\vspace{-2mm}
    \caption{\small Rational approximations of the transfer function for the hyperbolic PDE in Example~\ref{subsec:example_hyperbolic_PDE}.
    The solid black line represents the transfer function of the PDE. Colored lines represent a rational approximation for $r = 5$ and $10$. The dashed lines are the errors between the exact transfer function and the rational approximation.}
    \label{fig:HYPPDE}\vspace{-5mm}
\end{figure}

\subsection{PDE with spatially varying coefficients}\label{subsec:example_PDE_with_spatially_varying_coefficients}

The goal of this example is to demonstrate the universality of the proposed approach. Specifically, we consider 
the system $G$ modeled by a PDE with spatially varying coefficients\vspace{-1.5mm}
\begin{align*}
\partial_t \mbf v(t, \theta) &= a(\theta)\partial_\theta^2 \mbf v(t, \theta) + b(\theta)\partial_\theta \mbf v(t, \theta) + c(\theta)  \mbf v(t, \theta) + u(t)\\
\mbf v(t, 0) &= \partial_\theta \mbf v(t, 1) = 0 \qquad\qquad
y(t) = \mbf v(t, 1),\\[-7mm]
\end{align*}
where $a(\theta) = \theta^3 + 2 - \theta^2$, $b(\theta) = \partial_\theta a(\theta)$ and\vspace{-1.5mm}
\begin{equation*}
c(\theta) = -0.5\theta^3+1.3\theta^2 - 1.5\theta + 3.03.\vspace{-1.5mm}
\end{equation*}

Although the PDE is exponentially stable~\cite{peet2026hierarchical}, it is difficult to obtain the analytic transfer function of the PDE. Thus, in this example, we evaluate the approximation by computing the $H_\infty$ norm for the system difference. Specifically, using Thm.~\ref{thm:transfer_function_error_bounds}, we construct a $4^{th}$ order rational approximation  ($G_4$) with the state space representation as follows\vspace{-1.5mm}
\begin{align*}
    \dot x(t) \hspace{-0.5mm}&=\hspace{-1.5mm} \scriptsize\bmat{-21.68 \hspace{-0.5mm}&\hspace{-0.5mm}  -71.66 \hspace{-0.5mm}&\hspace{-0.5mm}  -81.67 \hspace{-0.5mm}&\hspace{-0.5mm}  -57.92 \\
  -72.53 \hspace{-0.5mm}&\hspace{-0.5mm} -271.37 \hspace{-0.5mm}&\hspace{-0.5mm} -304.08 \hspace{-0.5mm}&\hspace{-0.5mm} -219.87 \\
  -80.42 \hspace{-0.5mm}&\hspace{-0.5mm} -300.91 \hspace{-0.5mm}&\hspace{-0.5mm} -432.13 \hspace{-0.5mm}&\hspace{-0.5mm} -289.24 \\
  -58.23 \hspace{-0.5mm}&\hspace{-0.5mm} -217.90 \hspace{-0.5mm}&\hspace{-0.5mm} -312.92 \hspace{-0.5mm}&\hspace{-0.5mm} -376.49 \\
}\hspace{-1mm}x(t) \hspace{-0.5mm}+\hspace{-1mm}\bmat{ -118.64\\
 -396.76\\
 -439.95\\
 -318.58}\hspace{-1mm}u(t)\\
  y(t) \hspace{-0.5mm}& = \hspace{-0.5mm}\scriptsize\bmat{-0.0970 &    0.0313 &   -0.0055  &   0.0022} x(t)\\[-7mm]
\end{align*}

Using PIETOOLS~\cite{shivakumar2022dual}, we were able to verify that  $\|G - G_4\|_{H\infty} \leq 0.0108$. The error bounds in Cor.~\ref{cor:finite_Hinf_bounds} give us $\sup_{|\omega| \leq 5} |\hat G(j\omega) - \hat G_4(j\omega)| \leq 1.1$. The actual error may be significantly lower than the worst case errors in Cor.~\ref{cor:finite_Hinf_bounds}. Therefore, additional methods as in~\cite{shivakumar2022dual} may be required to compute the $H_\infty$ error bounds.

Note that in this case, we do not have analytic inversion of $\mcl A - \sigma \mcl T$ in the PIE formulation. Thus, the error bound may be affected by the error in the inverse of $\mcl A - \sigma \mcl T$.

\subsection{Comparison with Relevant Methods}\label{subsec:example_hyperbolic_PDE}
\begin{table}[t]\vspace{2mm}
    \begin{tabular}{c|c|c|c|c|c}
    & $r = 5$ & $r = 10$ & $r = 15$ & $r = 20$ & Discretization\\
    \hline
         KS  & $0.019$ & $0.006$ & $0.0029$ & $0.0017$ & none\\
         POD & $0.011$ &$0.009$ & - & - & $100$ nodes \\ 
         POD & $0.0075$ &$0.0020$ & - & - & $1000$ nodes
    \end{tabular}\vspace{-2mm}
    \caption{\small Finite-frequency $H_\infty$ error bounds for reduced-order systems computed by  the POD approach and  the Krylov subspace (KS) model reduction. } \vspace{-5mm}
    \label{tab:placeholder}
\end{table}

In this example, we compare the proposed approach with the balanced Proper Orthogonal Decomposition (POD) approach proposed in~\cite{Singler2011BalancedPF}. Let us consider a system $G$ modeled by the hyperbolic PDE \vspace{-1.5mm}
\begin{align*}
\partial_t w(t, \theta) &= (0.4 \theta - 0.5)\partial_\theta w(t, \theta) + (1-\theta) u(t)\\[-1mm]
w(t, 0) &= 0 \qquad\qquad y(t) = \int_0^1\theta w(t, \theta)d\theta\\[-7mm]
\end{align*}
As shown in~\cite{Singler2011BalancedPF}, the analytic transfer of the PDE has the form\vspace{-5mm}
{\small
\[
\quad \hat G(s) = \int_0^1 d\theta \int_0^\theta d\nu \frac{\theta(0.5-0.4\theta)^{s/0.4} (1 - \nu)}{(0.5 - 0.4 \nu)^{1 + s/0.4}}.\vspace{-1.5mm}
\]}
 The POD approach uses a discretization approach to compute the balanced realization of an infinite-dimensional system. Specifically, using 100 spatial nodes, the POD algorithm finds a $10^{th}$-order system $H_{10}$ such that $\|G - H_{10}\|_{H_\infty} \leq 0.009$. The Krylov subspace model reduction gives us a $10^{th}$-order system $G_{10}$ such that $\|G - G_{10}\|_{H_\infty}\leq 0.006$. For a $5^{th}$-order approximation, the POD approach using $100$ spatial nodes finds a $5^{th}$-order finite-dimensional system $H_5$ such that $\|G - H_{5}\|_{H_\infty} \leq 0.011$. We obtain a $5^{th}$-order system  $\|G - G_{5}\|_{H_\infty}\leq 0.019$.
The POD error bounds have been evaluated numerically for $\omega \in[10^{-3}, 10^2]$.  The comparison is presented in Tab.~\ref{tab:placeholder}.

 The Krylov subspace method processes the PIE system directly and does not require any discretization steps. Thus, the reduced-order system is stable by construction and admits theoretical error bounds. Although the POD approach may achieve lower error bounds (as shown in Tab.~\ref{tab:placeholder}), the approach requires high-order discretization, which can be difficult to achieve for complex systems.  


\subsection{Robust Stability Analysis of the PDE with Uncertainty}
\label{subsec:example_Diffusion}
\begin{figure}
    \centering\vspace{2mm}
    \includegraphics[width=0.7\linewidth]{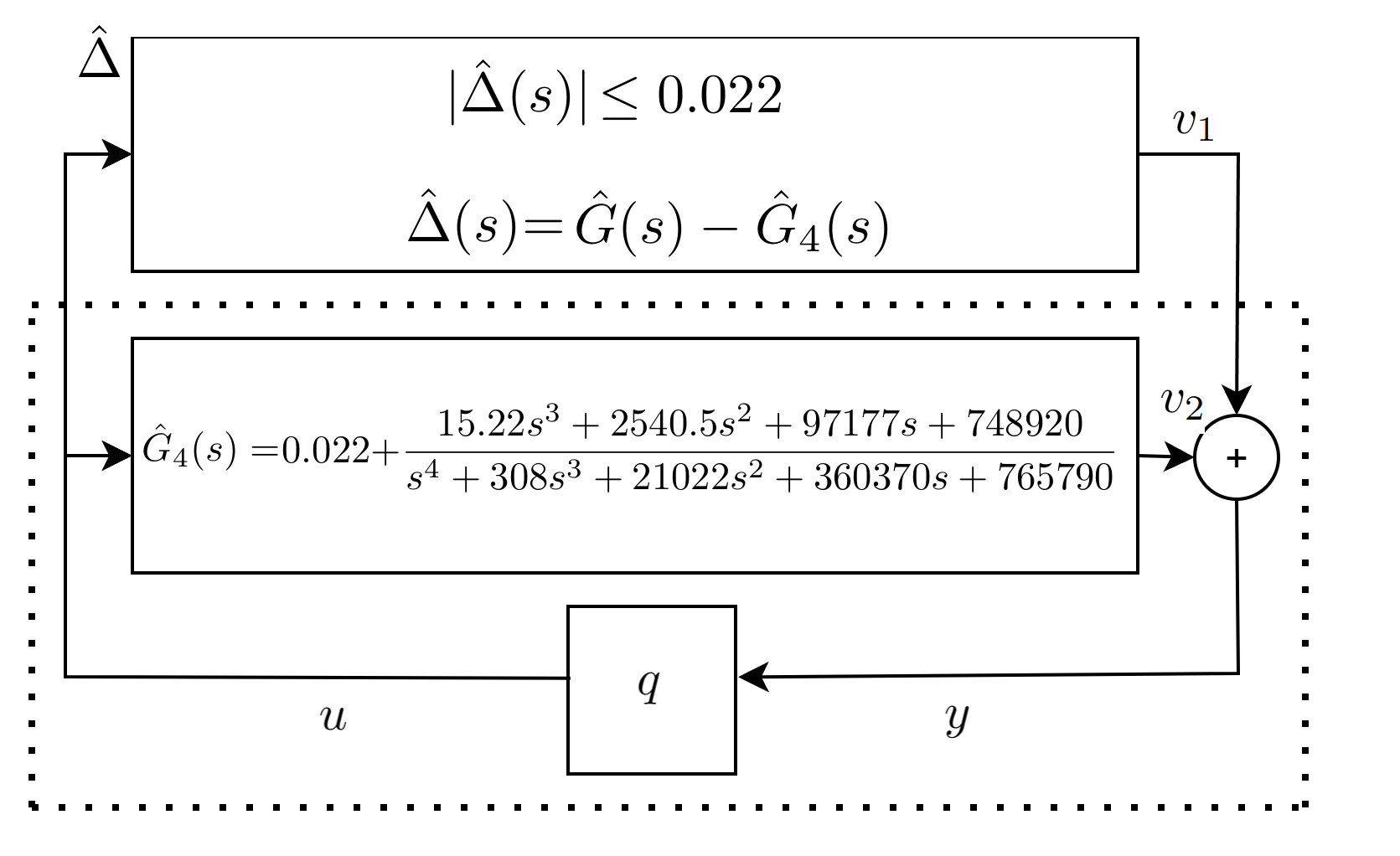}\vspace{-2mm}
    \caption{\small The feedback interconnection representation of the PDE in~\eqref{eqn:PDE_example_Diffusion} for stability analysis. The system $G$ is modeled by the PDE as in Eq.~\eqref{eqn:PDE_approx}. $G_{4}$ is the rational approximation obtained as in Eqn.~\eqref{eqn:diffusion_equation_rational_approximation}. The parameter $q$ is the uncertainty $|q| <\alpha$. The dashed line represents the finite-dimensional block in the interconnection.}
    \label{fig:interconnection}\vspace{-5mm}
\end{figure}
In this example, we apply the Krylov subspace method for robust stability analysis of the diffusion equation with uncertainty in the Robin boundary condition. Specifically, we consider\vspace{-1.5mm}
\begin{align}\label{eqn:PDE_example_Diffusion}
\partial_t \mbf v(t, \theta)& = \partial_\theta^2 \mbf v(t, \theta) \\
\mbf v(t, 0)& = 0,\qquad  \partial_\theta\mbf v(t, 1) - q\mbf v(t, 1) = 0\notag\\[-7mm]\notag
\end{align}
where $q$ is an uncertain parameter, $|q| \leq \alpha$. 
Our goal is to find the largest $\alpha$ such that the system 
is robustly stable.


For this task, we consider a feedback interconnection of the form\vspace{-3mm}
\begin{align*}
\hat{v}_1(s) &= \hat \Delta(s)\hat u(s) := (\hat G(s)-\hat G_{4}(s))\hat{u}(s)\\
    \hat{y}(s) &= \hat G_{4}(s) \hat{u}(s) + \hat{v}_1(s), \quad \hat{u}(s) =  q \hat{y}(s), \\[-7mm]
\end{align*}
where the system $G$ is modeled by the PDE\vspace{-1.5mm}
\begin{align}\label{eqn:PDE_approx}
\partial_t \hat{\mbf w}(t, \theta)& = \partial_\theta^2 \hat{\mbf w}(t, \theta) \qquad v_1(t) = \hat{\mbf w}(t, 1) \\
\hat{\mbf w}(t, 0)& = 0,\qquad  \partial_\theta\hat{\mbf w}(t, 1) = u(t).\notag\\[-7mm]\notag
\end{align}
$G_{4}$ is the $4^{th}$-order approximation of the system $G$ and $q$ is the uncertain parameter. $\hat \Delta(s)$ represents the truncation error introduced by the rational approximation.

As shown in~\cite{curtain2009transfer}, the transfer function of the PDE in~\eqref{eqn:PDE_approx} is $\hat G(s) = \tanh(\sqrt{s})/\sqrt{s}$. The PIE representation of the system in~\eqref{eqn:PDE_approx} has an input $\dot u(t)$. Thus, to obtain the transfer function of the PDE in~\eqref{eqn:PDE_approx}, we approximate $s \mcl C(s\mcl T - \mcl A)^{-1} \mcl B + \mcl D$. A slight variation of Thm.~\ref{thm:transfer_function_error_bounds} gives us a $4^{th}$-order approximation $G_{4}$  of $G$. However, since the conditions of Cor.~\ref{cor:finite_Hinf_bounds} are not satisfied, we evaluate the error bounds numerically at $\omega \in [10^{-10}, 10^{10}]$, which gives us  $|\hat \Delta(s)| = |\hat G(s) - \hat G_{4}(s)| \leq 0.022$.

For reference, the rational approximation has the form\vspace{-1mm}{\footnotesize
\begin{equation}\label{eqn:diffusion_equation_rational_approximation}
\hat G_4 (s) = 0.022 + \frac{  15.22 s^3 + 2540.5 s^2 + 97177 s + 748920 }{s^4 + 308 s^3 + 21022 s^2 + 360370 s + 765790}.
\end{equation}}
Using the LFT transformation, we obtain\vspace{-1.5mm}
\[
\hat v_1(s) = \hat \Delta(s) \hat u(s), \quad \hat u(s) = \frac{q}{1 - q\hat G_4(s)} \hat v_1(s).\vspace{-1.5mm}
\]

 The MATLAB subroutine \texttt{wcgain} shows that\vspace{-1.5mm}
\[
\sup_{\omega\in\R}\big|\frac{q}{1-\hat G_{4}(j\omega) q} \big| \leq 44.5471 < \frac{1}{0.022}\vspace{-1.5mm} \quad \text{ if $|q| \le 0.978$}
\]

Thus, by the small-gain theorem, we conclude the stability of the feedback interconnection and the PDE in~\eqref{eqn:PDE_example_Diffusion} for $q \in [-0.978, 0.978]$.  The higher-order approximation may improve the stability interval. For $16^{th}$ order, we obtain the stability interval $q \in [-0.99, 0.99]$. The theoretical analysis shows that the PDE is exponentially stable for $q < 1$.

\section{Conclusion}
 In this paper, we addressed the problem of computing a rational approximation of transfer functions for PDEs.   For the PIE representation, we constructed a truncated controllability matrix and projected the PIE states onto the vector space spanned by the truncated controllability matrix. The properties of the truncated controllability matrix allowed us to propose a rational approximation of the PIE transfer function. Using sufficient conditions for the exponential stability and well-posedness of the PIE system, we showed that the proposed approximation is stable and converges to the PIE transfer functions. Then, we demonstrated equivalence between the PIE and PDE transfer functions. Finally, the performance of the method was verified using simple numerical examples.




\addtolength{\textheight}{-12cm}   







\bibliographystyle{IEEEtran}
\bibliography{TransferFunctions}



\end{document}